\documentclass[10pt]{amsart}
\usepackage{calrsfs}

\usepackage{amsmath,tikz}
\usetikzlibrary{tikzmark,fit}

\usepackage[T1]{fontenc}

\usepackage{graphicx}
\usepackage{amsthm}
\usepackage{amssymb}
\usepackage{multirow}
\usepackage{tikz-cd}
\usepackage{amsmath}
\usepackage{todonotes}
\usepackage{amsbsy}
\usepackage[all]{xy}
\usepackage{qtree}

\usepackage{enumitem}
\usepackage{xfrac}    
\usepackage{color, colortbl}
\definecolor{LightCyan}{rgb}{0.88,1,1}
\definecolor{Gray}{gray}{0.9}

\usepackage{faktor}

\usepackage{changes}
\definechangesauthor[color=orange,name={Aristidesd Kontogeorgis}]{AK}
\definechangesauthor[color=blue,name={Alex Terezakis}]{AT}

\newtheorem{theorem}{Theorem}
\newtheorem{lemma}[theorem]{Lemma}
\newtheorem*{lemma*}{Lemma}

\newtheorem{proposition}[theorem]{Proposition}
\newtheorem*{proposition*}{Proposition}
\theoremstyle{definition}

\newtheorem*{remark}{Remark}

\newcommand{\Gal}{{ \mathrm Gal}}
\renewcommand{\mod}{{\;\mathrm{mod}}}

\date{\today}

\title{On the de Rham cohomology of cyclic covers}

\author[A. Kontogeorgis]{Aristides Kontogeorgis}
\address{Department of Mathematics, National and Kapodistrian  University of Athens
Pane\-pist\-imioupolis, 15784 Athens, Greece}
\email{kontogar@math.uoa.gr}

\author[O. Lygdas]{Orestis Lygdas}
\address{Department of Mathematics, National and Kapodistrian  University of Athens
Pane\-pist\-imioupolis, 15784 Athens, Greece}
\email{olygdas@math.uoa.gr}

\date \today

\makeatletter
\newcommand{\aprod}{\mathop{\operator@font \hbox{\Large$\ast$}}}
\makeatother

\begin{document}



\begin{abstract}
We compute explicit bases for the de~Rham cohomology of cyclic covers of the
projective line defined over an algebraically closed field of characteristic $p\geq 0$.  For both Kummer and Artin–Schreier extensions, we describe
precise $k$-bases for the cohomology groups $H^{1}(X,\mathcal{O}_{X})$ and
$H^{0}(X,\Omega_{X})$, and we use these to construct an explicit basis for the
first de~Rham cohomology group $H^{1}_{\mathrm{dR}}(X/k)$ via Čech cohomology.
Our approach relies on detailed computations of divisors of functions and
differentials, together with residue calculations and the duality pairing
between $H^{0}(X,\Omega_{X})$ and $H^{1}(X,\mathcal{O}_{X})$.  The resulting
expressions are given in closed form in terms of the defining equation of the
cover, making the cohomology fully explicit and readily applicable to questions
involving group actions, Hodge--de~Rham filtrations, and the study of $p$-cyclic
covers.
\end{abstract}
\maketitle

\section{Introduction}

There are few classes of curves where explicit computations can be made. Cyclic covers of curves which generalize hyperelliptic curves are among them. H. Hasse in \cite{Hasse34} provided a unified treatment of both Artin-Schreier and Kummer extensions of the rational function fields and H. Boseck \cite{Boseck} gave explicit bases of holomorphic differentials on them. 
Furthermore, the work of Sekiguchi, Oort, and Suwa
\cite{MR1011987},
\cite{SeSu94}, \cite{SeSu95} 
 highlights deeper
connections between Kummer and Artin–Schreier-type extensions via group scheme
theory, suggesting that a unified framework for these computations is possible, \cite{1905.05545},\cite{KaranProc}.

On the other hand there is a lot of interest concerning study of the de~Rham cohomology of algebraic curves over fields of arbitrary characteristic especially with respect to their Galois module structure and the splitting problem of the exact sequence 
\[
	1 \rightarrow H^0(X,\Omega_X) 
	\rightarrow  H^1_{\mathrm{dR}}(X) \rightarrow  H^1(X, \mathcal{O}_X) \rightarrow  0, 
\]
seen as Galois module, see 
\cite{MR4363802},
\cite{MR4630761}, 
\cite{MR4961127},
\cite{MR4881009},  
\cite{MR3782449}.  
For cyclic covers of the projective line, the geometry and arithmetic of the
curve are encoded in an explicit equation, and this allows one to perform
cohomological computations in a concrete manner.  
Nevertheless, even in this relatively accessible class of curves, an explicit
description of the Hodge–de~Rham exact sequence and the associated bases for
$H^1_{\mathrm{dR}}(X/k)$ remains delicate.


The aim of this article is to give explicit $k$-bases for the cohomology groups
$H^1(X,\mathcal{O}_X)$, $H^0(X,\Omega_X)$, and the first de~Rham cohomology group
$H^1_{\mathrm{dR}}(X/k)$ when $X$ is a cyclic cover of the projective line of either
Kummer or Artin--Schreier type.  
Our approach relies on Čech cohomology with respect to the natural affine
covering induced by the projection $X\to\mathbb{P}^1$, combined with explicit
computations of divisors of functions and differentials.  This method allows us
to describe bases in completely elementary terms, making the cohomology
transparent in examples and suitable for applications to deformation theory,
$G$-actions on cohomology, and the study of the Hodge–de~Rham sequence.

The structure of the paper is as follows.  
In Section \ref{sec:Kumm} we treat the case of Kummer extensions, beginning with divisor
computations, then exhibiting bases for $H^1(X,\mathcal{O}_X)$ and
$H^0(X,\Omega_X)$, and finally constructing an explicit basis for
$H^1_{\mathrm{dR}}(X/k)$.  
Section \ref{sec:AS} contains the analogous results for Artin–Schreier extensions, where
the arithmetic of characteristic $p$ introduces new phenomena.  
Throughout both sections, duality pairings and residue computations play a
central role.








\medskip
\noindent {\bf Aknowledgements} The research project is implemented in the framework of H.F.R.I Call “Basic research Financing (Horizontal support of all Sciences)” under the National Recovery and Resilience Plan “Greece 2.0” funded by the European Union Next Generation EU(H.F.R.I.  
Project Number: 14907.
\begin{center}
\includegraphics[scale=0.4]{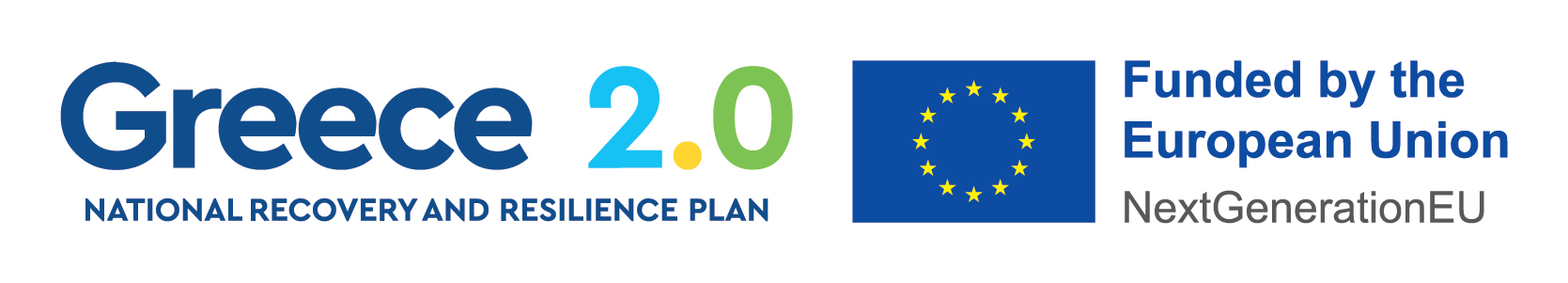}
\hskip 1cm
\includegraphics[scale=0.05]{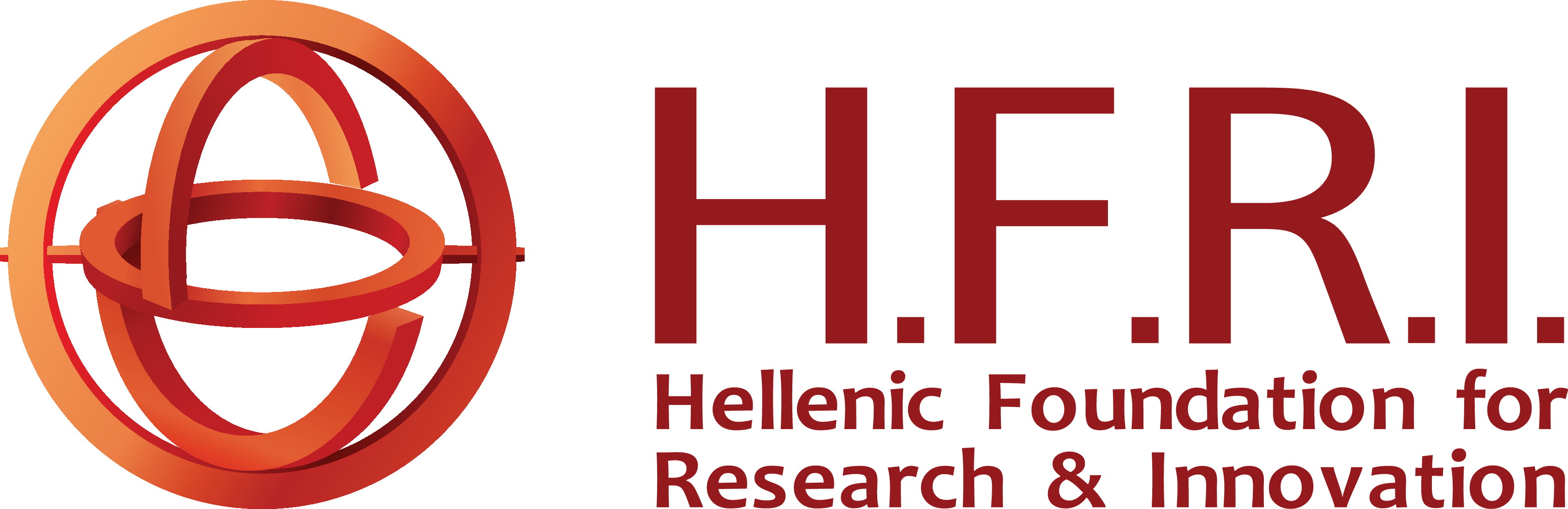}
\end{center}

\section{Kummer extensions}
\label{sec:Kumm}
\subsection{Computations}
We consider the curve given as a Kummer extension of $k(x)$, with $k$ an algebraically closed field of characteristic $p>0$ given by an equation of the form 
\[
	X:y^n=f(x)=\prod_{i=1}^r (x-\rho_i)^{l_i}.
\]
We recall some facts about these curves, see \cite{Boseck}, \cite{Ko:99}, \cite{Ko:98}. 
Without loss of generality we can assume that $l := \deg f(x)\equiv 0 \pmod {n}$. Thus we get that the points above $\infty$ are not ramified in the cover $X \rightarrow \mathbb{P}^1$. The $\rho_i$'s are the only points of $\mathbb{P}^1$ in the branch locus of the cover, and each point above $x=\rho_i$ is ramified with ramification index $e_i= \frac{n}{(n,l_i)}$. The points of $X$ lying above each $\rho_i$ are 
$P_{i,1},\ldots,P_{i,g_i}$, where $g_i=(n,l_i)$, while the points above $\infty$ are $P_{\infty,1},\ldots,P_{\infty,n}$.\par
If $0\in \mathbb{P}^1$ is in the branch locus, then its corresponding exponent in the equation of $X$ will also be denoted as $l_0$, the number of points above it equals $g_0=(n,l_0)$, and their ramification index is $e_0=\frac{n}{(n,l_0)}$. If $0$ is not in the branch locus, then $l_0:=n,g_0:=n,\text{ and } e_0:=1$.\par
We compute 
\begin{align*}
(x) &= e_0\sum_{j=1}^{g_0} P_{0,j} - \sum_{m=1}^n P_{\infty,m}
\\
(y) &= \sum_{i=1}^r \lambda_i\sum_{j=1}^{g_i} P_{i,j} - t\sum_{m=1}^n P_{\infty,m}
\end{align*}
where 
\[
  t=l/n, \text{ and }\lambda_i = \frac{e_i l_i}{n}= \frac{l_i}{(l_i,n)}.
\]
For each $\mu$, $\mu \in \{1,2,\ldots,n-1\}$, we define $m_i^{(\mu)}$, $\upsilon_i^{(\mu)}$ by
\[
  \mu \lambda_i = \mu \frac{l_i}{(l_i,n)}= m_i^{(\mu)} e_i + \upsilon_i^{(\mu)}, \text{ with } 0 \leq \upsilon_i^{(\mu)}  \leq e_i-1.
\]
We set
\begin{equation}
	\label{eq:gmdef}
g_\mu(x) = \prod_{i=1}^r (x-\rho_i)^{m_i^{(\mu)}} \text{ for } \mu \in \{1,2,\ldots,n-1\}
\end{equation}
and we compute 
\begin{equation}
\label{eq:extrael}
  \left( \frac{y^{\mu}}{g_\mu(x)} \right) = \sum_{i=1}^r \upsilon_i^{(\mu)} \sum_{j=1}^{g_i} P_{i,j} 
  - t^{(\mu)} \sum_{m=1}^n P_{\infty,m}, 
\end{equation}
where 
\[
  t^{(\mu)}=\frac{1}{n} \sum_{i=1}^r g_i \upsilon_i^{(\mu)}.
\]
%
%
%
\subsection{A $k$-vector space bases for $H^1(X, \mathcal{O}_X)$, $H^0(X,\Omega_X)$}
We begin by computing a basis for $H^1(X,\mathcal{O}_X)$. 
We denote by $\pi:X \rightarrow \mathbb{P}^1$ a covering map from $X$ to $\mathbb{P}^1$. For every point $a \in \mathbb{P}^1$ we define $U_a := X \backslash \pi^{-1}(a)$.
Observe that $x^iy^j$ for $i,j \geq 0$ are elements of $\mathcal{O}_X(U_\infty)$. Also by equation (\ref{eq:extrael}) the elements 
$
\frac{y^{\mu}}{g_\mu(x)}
$
are in $\mathcal{O}_X(U_\infty)$, while the elements 
\[
 \frac{1}{x^{\nu}} \frac{y^\mu}{g_\mu(x)}  \text{ for } \nu \geq t^{(\mu)}  
 \text{ and } \frac{1}{x^{\nu}} \text{ for } \nu \geq 0,
\]
are in $\mathcal{O}_X(U_0)$. We identify the first sheaf cohomology group of $\mathcal{O}_X$ with the corresponding first Čech cohomology group, so  
\[
  H^{1}(X, \mathcal{O}_X) = \frac{  \mathcal{O}_X(U_0 \cap U_\infty) }{
  \{ f_0- f_\infty: 
f_0 \in \mathcal{O}_X(U_0), f_\infty \in \mathcal{O}_X(U_\infty) \}
  }.
\]
The $k$-vector space $H^1(X, \mathcal{O}_X)$  admits as basis the following  linear equivalence  classes
\[
h_{\mu,\nu} := \left[
\frac{1}{x^{\nu}} \frac{y^\mu}{g_\mu(x)}
  \right], 1\leq \mu \leq n-1 \text{ s.t } t^{(\mu)}\geq 2, 1 \leq \nu \leq t^{(\mu)}-1.
\]
The number of these classes equals 
\[
  \sum_{\mu=1}^{n-1} (t^{(\mu)}-1), 
\]
which are equal to the genus of the curve $X$ according to \cite[eq. (47)]{Boseck}.\par
We will now give a basis for the space  $H^0(X,\Omega_X)$. 
By \cite[Satz 16]{Boseck} we get that a basis of holomorphic differentials is given by 
\begin{equation}
	\label{eq:omegadef}
\omega_{\mu,\nu} := x^{\nu-1} \frac{g_\mu(x)}{y^\mu} dx, \mu\in \{1,\ldots,n-1\} \text{ s.t. } t^{(\mu)}\geq 2, 1 \leq \nu \leq t^{(\mu)}-1.
\end{equation}
For our description of $H^1(X,\mathcal{O}_X)$ using Čech cohomology, we can see from Tait's dissertation \cite[sec. 2.5]{TaitPhd} that Serre's duality map becomes (the coefficient is not essential to the definition of the map, but will be useful later)
\begin{align*}
\langle\cdot,\cdot\rangle:H^0(X,\Omega_X) \times H^1(X, \mathcal{O}_X) & \longrightarrow  k \\
\langle\omega, [f]\rangle & \longmapsto  \left(-\frac{1}{n}\right)\sum_{P\in \pi^{-1}(\infty)}\mathrm{Res}_P(f \omega),
\end{align*}
where $\mathrm{Res}_P(\omega)$ is the residue of the differential $\omega$ at the point $P\in X$  \cite[Chap. II, 7,8,12]{SerreAlgGroupsClassFields}. We want to show that $\{h_{\mu,\nu}\}_{\mu,\nu}$ is the dual basis to $\{\omega_{\mu,\nu}\}_{\mu,\nu}$, i.e. 
\[
\langle \omega_{\mu_1,\nu_1}, h_{\mu_2,\nu_2} \rangle =
\delta_{\mu_1,\mu_2} \delta_{\nu_1,\nu_2}.
\]
We will first recall the definition of the residue of a differential. 
Let $F$ be the function field of $X$, denote by $\widehat{F_P}=k((t))$ the completion of $F$ for the valuation $v_P$, where $t$ is a local uniformizer at $P$. We identify $f \in F$ with its image in $\widehat{F_P}$ which is called the Laurent expansion of $f$ at $P$:
$$f = \sum_{i>>-\infty} c_i t^i, c_i\in k.$$
The residue $\mathrm{Res}_P(\omega)$ 
is defined as the coefficient $c_{-1}$ of the Laurent expansion of $f$ in a decomposition $\omega = f dt$. 
\par 
Some important properties for the residue can be found in no.11 of \cite{SerreAlgGroupsClassFields}:
\begin{enumerate}
	\item $\mathrm{Res}_P(\omega)$ is $k$-linear in $\omega$.
	\item $\mathrm{Res}_P(\omega)=0$ if $v_P(\omega)\geq 0$.
	\item $\mathrm{Res}_P(dg) = 0$ for all $g\in \widehat{F_P}$.
	\item $\mathrm{Res}_P(dg/g)=v_P(g)$ for all $g\in \widehat{F_P}\setminus\{0\}$.
\end{enumerate}
We will also need the following 
 lemma \cite[lemma 4 no. 12]{SerreAlgGroupsClassFields}:
\begin{lemma}
	If we have a covering map $\varphi: X \rightarrow \mathbb{P}^1$ of curves, then for every point $P\in \mathbb{P}^1$,
	\[
	\sum_{Q\rightarrow P} \mathrm{Res}_Q(\omega) = \mathrm{Res}_P(\text{Tr}(\omega)),
	\]
	the sum being over all the points $Q\in X$ such that $\varphi(Q)=P$.
\end{lemma}
In our case we have the covering map 
\begin{align*}
\pi:& X\rightarrow \mathbb{P}^1\\
&(x,y)\mapsto x,
\end{align*}
which corresponds to the Galois extension of function fields $F/E$, where $E=k(x)$ is the function field of $\mathbb{P}^1$. 
If  
$\omega = f dx$, then 
$$\text{Tr}(\omega)=\text{Tr}_{F/E}(\omega) = \text{Tr}_{F/E}(f)dx.$$
Therefore we get
\[
\sum_{m=1}^n \mathrm{Res}_{P_{\infty,m}}(\omega) = \mathrm{Res}_{\infty}(\text{Tr}(\omega)) = \mathrm{Res}_{\infty}(\text{Tr}_{F/E}(f)dx).
\]
\begin{proposition}
	The bases given above for $H^0(X,\Omega_X)$ and $H^1(X,\mathcal{O}_X)$ are dual to each other.
\end{proposition}
\begin{proof}
We calculate
\begin{align*}
\langle \omega_{\mu_1,\nu_1}, h_{\mu_2,\nu_2} \rangle =& -\frac{1}{n}\sum_{P\rightarrow \infty} \mathrm{Res}_P\left(\frac{x^{\nu_1-\nu_2}y^{\mu_2-\mu_1}g_{\mu_1}(x)}{g_{\mu_2(x)}}\frac{dx}{x}\right)\\
=& -\frac{1}{n}\mathrm{Res}_{\infty}\left(\text{Tr}_{F/E}\left(\frac{x^{\nu_1-\nu_2}y^{\mu_2-\mu_1}g_{\mu_1}(x)}{g_{\mu_2(x)}}\frac{dx}{x}\right)\right).
\end{align*}\par
If $\mu_1=\mu_2, \nu_1=\nu_2$, then
$$\langle \omega_{\mu_1,\nu_1}, h_{\mu_2,\nu_2} \rangle = -\frac{1}{n} \mathrm{Res}_{\infty}\left(n\frac{dx}{x}\right) = \mathrm{Res}_{\infty}\left(\frac{d\left(\frac{1}{x}\right)}{\frac{1}{x}}\right) = 1,$$
using the properties of the residue mentioned above.\par
Suppose now that either $\mu_1\neq \mu_2$ or $\nu_1\neq \nu_2$. Since the extension is Galois and only $y$ is not in the base field of the extension, we have
$$\text{Tr}_{F/E}\left(\frac{x^{\nu_1-\nu_2-1}y^{\mu_2-\mu_1}g_{\mu_1}(x)}{g_{\mu_2(x)}}\right) = \frac{x^{\nu_1-\nu_2-1}g_{\mu_1}(x)}{g_{\mu_2(x)}} \text{Tr}_{F/E}(y^{\mu_2-\mu_1}).$$
Since we are in the Kummer extension case, $\Gal(E/F) = \langle\sigma\rangle$ where
\begin{align*}
\sigma: F&\rightarrow F\\
x&\mapsto x\\
y&\mapsto y\cdot \zeta_n
\end{align*}
where $\zeta_n$ is a primitive $n$-th root of unity. Thus
\begin{equation} 
	\label{sum:zeta}
\text{Tr}_{F/E}(y^{\mu_2 -\mu_1}) = \sum_{i=0}^{n-1} \sigma^i(y^{\mu_2-\mu_1}) = y^{\mu_2-\mu_1} \sum_{i=0}^{n-1} \zeta_n^{i(\mu_2-\mu_1)}.
\end{equation}
If $\mu_1=\mu_2$ and $\nu_1\neq \nu_2$, then 
$$\langle \omega_{\mu_1,\nu_1}, h_{\mu_2,\nu_2} \rangle = -\frac{1}{n}\mathrm{Res}_{\infty}\left(nx^{\nu_1-\nu_2-1}dx\right).$$
If $\nu_1>\nu_2$, then $\nu_1-\nu_2-1\geq 0$ and
\begin{align*}
\mathrm{Res}_{\infty}\left(nx^{\nu_1-\nu_2-1}dx\right) =& \mathrm{Res}_{\infty}\left(-n\frac{1}{x^{-(\nu_1-\nu_2-1)}}x^2d\left(\frac{1}{x}\right)\right)\\
=&\mathrm{Res}_{\infty}\left(-n\frac{1}{x^{-(\nu_1-\nu_2+1)}}d\left(\frac{1}{x}\right)\right)
\end{align*}
since $d\left(\frac{1}{x}\right) = -\frac{1}{x^2}dx$. Since  $\nu_1-\nu_2+1\geq 2$, from the definition of the residue we see that 
$$\mathrm{Res}_{\infty}\left(-n\frac{1}{x^{-(\nu_1-\nu_2+1)}}d\left(\frac{1}{x}\right)\right) = 0.$$
If $\nu_1<\nu_2$, then working similarly we see that again $\langle \omega_{\mu_1,\nu_1}, h_{\mu_2,\nu_2} \rangle =0$.\par
If $\mu_1\neq \mu_2$, then the sum of roots of unity appearing in eq. (\ref{sum:zeta}) is known to be zero.


In conclusion, 
$$\langle \omega_{\mu_1,\nu_1}, h_{\mu_2,\nu_2} \rangle = 0$$
when either $\mu_1\neq \mu_2$ or $\nu_1\neq \nu_2$.
\end{proof}
\subsection{A $k$-vector space basis for $H^1_{dR}(X/k)$}
Using the same description for the first de Rham cohomology group of $X$ (via Čech cohomology) as the one given in \cite{MR3782449} we have that $H^1_{dR}(X/k)$ can be calculated by the quotient
\begin{center}
\resizebox{.99\hsize}{!}{$\frac{\{(\omega_0,\omega_{\infty},f_{0\infty}) \in \Omega_X(U_0) \times \Omega_X(U_{\infty}) \times \mathcal{O}_X(U_0\cap U_{\infty}) \mid df_{0\infty}=\omega_0|_{U_0\cap U_{\infty}} - \omega_{\infty}|_{U_0\cap U_{\infty}} \} }{
		\{ (df_0, df_{\infty}, f_0|_{U_0\cap U_{\infty}} - f_{\infty}|_{U_0\cap U_{\infty}}) \mid f_0\in \mathcal{O}_X(U_0), f_{\infty}\in \mathcal{O}_X(U_{\infty})\}}.$}	
\end{center}
There are also canonical maps
\begin{align*}
i: H^0(X,\Omega_X) & \rightarrow H^1_{dR}(X/k) \\
\omega & \mapsto [(\omega|_{U_0},\omega|_{U_{\infty}},0)]
\end{align*}
and
\begin{align*}
p: H^1_{dR}(X/k) & \rightarrow H^1(X,\mathcal{O}_X) \\
[(\omega_0, \omega_{\infty}, f_{0\infty})] & \mapsto [f_{0\infty}].
\end{align*}
From Proposition 4.1.2 of Tait's dissertation \cite{TaitPhd} we get that the following short sequence is exact (the result is about hyperelliptic curves, but it can be extended with no changes in the proof in our setting):
\begin{equation*}
0\rightarrow H^0(X,\Omega_X) \xrightarrow{i} H^1_{dR}(X/k) \xrightarrow{p} H^1(X,\mathcal{O}_X) \rightarrow 0,
\end{equation*}
hence the basis of the $k$-vector space $H^1_{dR}(X/k)$ can be calculated from the bases of $H^0(X,\Omega_X)$ and $H^1(X,\mathcal{O}_X)$.\par
In order to write down an explicit $k$-basis for the first de Rham cohomology group, we will need the following notation: Say we have $h(x)\in k[x]$ and $m\in \mathbb{N}$, then $h^{\leq m}(x)$ (respectively $h^{>m}(x)$) is the sum of monomials of $h(x)$ with degree $\leq m$ (respectively $>m$).
\begin{theorem}\label{Kummer_basis}
 Let $I^{(\mu)} \subset \mathbb{N}$ defined by 
$$I^{(\mu)}:= \{1\leq i\leq r \vert \upsilon_i^{(\mu)}\neq 0\}.$$
Set
$$\psi_{\mu,\nu}(x):= \sum_{i=1}^r \left(g_i\upsilon_i^{(\mu)}x\prod_{j\in I^{(\mu)}\setminus\{i\}}(x-\rho_j)\right) - \nu n\prod_{i\in I^{(\mu)}}(x-\rho_i)\in k[x].$$
Then the elements
\begin{gather*}
a_{\mu,\nu} := \left[\left(\frac{\psi_{\mu,\nu}^{\leq\nu+1}(x)}{nx^{\nu+1}}\frac{g_{n-\mu}(x)}{y^{n-\mu}}dx, \frac{\psi_{\mu,\nu}^{>\nu+1}(x)}{nx^{\nu+1}}\frac{g_{n-\mu}(x)}{y^{n-\mu}}dx, \frac{y^{\mu}}{x^{\nu}g_{\mu}(x)}\right)\right] \\
\text{ for } 1\leq \mu \leq n-1 \text{ s.t. } t^{(\mu)}\geq 2 \text{ and } 1\leq \nu \leq t^{(\mu)}-1
\end{gather*}
together with the elements
\begin{gather*}
\delta_{\mu,\nu} := \left[\left(\frac{x^{\nu-1}g_{\mu}(x)}{y^{\mu}}dx,\frac{x^{\nu-1}g_{\mu}(x)}{y^{\mu}}dx,0\right)\right] \\
\text{ for } 1\leq \mu \leq n-1 \text{ s.t. } t^{(\mu)}\geq 2 \text{ and } 1\leq \nu \leq t^{(\mu)}-1
\end{gather*}
form a basis for $H^1_{dR}(X/k)$ as a $k$-vector space.
\end{theorem}
\begin{proof}
It can be easily seen that the projections of the $a_{\mu,\nu}$'s form the basis of $H^1(X,\mathcal{O}_X)$ and that the unique preimages of the $\delta_{\mu,\nu}$'s form the basis of $H^0(X,\Omega_X)$. Therefore, if the $a_{\mu,\nu}$'s and $\delta_{\mu,\nu}$'s are well-defined elements of $H^1_{dR}(X/k)$, then they indeed form a $k$-basis for this space. The $\delta_{\mu,\nu}$'s are evidently well-defined, so we focus on the $a_{\mu,\nu}$'s.\par
First, we show that $\phi_{\mu}(x):=\left(\frac{y^{\mu}}{g_{\mu}(x)}\right)^n $ is a polynomial of $x$:
\begin{align*}
\phi_{\mu}(x) = \left(\frac{y^{\mu}}{g_{\mu}(x)}\right)^n & = \frac{f(x)^{\mu}}{g_{\mu}(x)^n} = \frac{\prod_{i=1}^{r} (x-\rho_i)^{\mu l_i}}{\prod_{i=1}^{r} (x-\rho_i)^{n m^{(\mu)}_i}}\\
& = \prod_{i=1}^{r} (x-\rho_i)^{\upsilon^{(\mu)}_i (n,l_i)}.
\end{align*}
Because we'll need it later we write here a convenient expression for the formal derivative of $\log(\phi_{\mu}(x))$:
\begin{gather}
	\label{eq:logder}
\phi'_{\mu}(x) = \left(\prod_{i=1}^{r} (x-\rho_i)^{\upsilon^{(\mu)}_i (n,l_i)}\right)' = \phi_{\mu}(x)\sum_{i=1}^{r}\frac{\upsilon^{(\mu)}_i (n,l_i)}{x-\rho_i}\\
\Rightarrow \frac{\phi'_{\mu}(x)}{\phi_{\mu}(x)} = \sum_{i=1}^{r}\frac{\upsilon^{(\mu)}_i (n,l_i)}{x-\rho_i} = \sum_{i\in I^{(\mu)}}\frac{\upsilon^{(\mu)}_i (n,l_i)}{x-\rho_i}. \nonumber
\end{gather}\par
Now, since we cannot do much to calculate the differential of $\frac{y^{\mu}}{x^{\nu}g_{\mu}(x)}$, we'll calculate the differential of its n-th power instead:
\begin{align*}
d\left(\left(\frac{y^{\mu}}{x^\nu g_\mu(x)} \right)^n \right) & = d\left(\frac{\phi_{\mu}(x)}{x^{\nu n}}\right)\\
& = \frac{\phi'_{\mu}(x) x^{\nu n} - \phi_{\mu}(x)\nu n x^{\nu n -1}}{x^{2\nu n}}dx\\
& = \frac{\phi'_{\mu}(x)}{x^{\nu n}}dx - \frac{\phi_{\mu}(x)\nu n}{x^{\nu n + 1}}dx.
\end{align*}
Also
\begin{gather*}
d\left(\left(\frac{y^{\mu}}{x^\nu g_\mu(x)} \right)^n \right) = n\left(\frac{y^{\mu}}{x^{\nu}g_{\mu}(x)}\right)^{n-1}d\left(\frac{y^{\mu}}{x^{\nu}g_{\mu}(x)}\right)
\end{gather*}
hence
$$n\left(\frac{y^{\mu}}{x^{\nu}g_{\mu}(x)}\right)^{n-1}d\left(\frac{y^{\mu}}{x^{\nu}g_{\mu}(x)}\right) = \frac{\phi'_{\mu}(x)}{x^{\nu n}}dx - \frac{\phi_{\mu}(x)\nu n}{x^{\nu n + 1}}dx$$
which implies
\begin{align} \nonumber
d\left(\frac{y^{\mu}}{x^{\nu}g_{\mu}(x)}\right) = & \frac{1}{n} \left(\frac{x^{\nu}g_{\mu}(x)}{y^{\mu}}\right)^{n-1}\left(\frac{\phi'_{\mu}(x)}{x^{\nu n}}dx - \frac{\phi_{\mu}(x)\nu n}{x^{\nu n + 1}}dx\right)
\\
\label{eq:derder}
= & \left(\frac{y^{\mu}}{g_{\mu}(x)}\right)\left(\frac{1}{nx^{\nu+1}}\right)\left(\frac{x\phi'_{\mu}(x)}{\phi_{\mu}(x)} - \nu n\right)dx
\\
\nonumber
= & \left(\frac{y^{\mu}\cdot y^{n-\mu}}{g_{\mu}(x)\cdot g_{n-\mu}(x)}\right)\left(\frac{1}{nx^{\nu+1}}\right)\left(\frac{x\phi'_{\mu}(x)}{\phi_{\mu}(x)} - \nu n\right)\left(\frac{g_{n-\mu}(x)}{y^{n-\mu}}dx\right)
\nonumber
\\
\nonumber
= & \frac{1}{nx^{\nu+1}}\left(\frac{f(x)}{g_{\mu}(x)\cdot g_{n-\mu}(x)}\right)\left(\frac{x\phi'_{\mu}(x)}{\phi_{\mu}(x)} - \nu n\right)\omega_{n-\mu,1},
\end{align}
where $\omega_{n-\mu,1}\in H^0(X,\Omega_X)$ is a  holomorphic differential as defined in eq. (\ref{eq:omegadef}).\par
From the definition of $g_{\mu}(x)$ in eq. (\ref{eq:gmdef}), we know that:
$$g_{\mu}(x)\cdot g_{n-\mu}(x) = \prod_{i=1}^r (x-\rho_i)^{m_i^{(\mu)}+m_i^{(n-\mu)}}.$$
For the exponents appearing in the above expression we have that for $i\in \{1,\dots,r\}$ and $\mu\in \{1,\dots,n-1\}$:
\begin{align*}
\mu \lambda_i =& m_i^{(\mu)} e_i + \upsilon_i^{(\mu)}\\
(n-\mu) \lambda_i =& m_i^{(n-\mu)} e_i + \upsilon_i^{(n-\mu)}
\end{align*}
and adding these two equations we get (remember $\lambda_i=\frac{l_i}{(n,l_i)}$ and $e_i=\frac{n}{(n,l_i)}$):
$$n\frac{l_i}{(n,l_i)} = (m_i^{(\mu)}+m_i^{(n-\mu)}) \frac{n}{(n,l_i)} + (\upsilon_i^{(\mu)}+\upsilon_i^{(n-\mu)}),$$
with $\upsilon_i^{(\mu)},\upsilon_i^{(n-\mu)} \in \{0,\dots,\frac{n}{(n,l_i)}-1\}$. Taking the last equation $\mod e_i$ gives that:
$$\upsilon_i^{(\mu)}+\upsilon_i^{(n-\mu)}\equiv 0\mod e_i.$$
Therefore:
$$\upsilon_i^{(\mu)} \equiv 0\mod e_i \Leftrightarrow \upsilon_i^{(n-\mu)} \equiv 0\mod e_i \Leftrightarrow \upsilon_i^{(\mu)} = 0 \Leftrightarrow \upsilon_i^{(n-\mu)} =0.$$ 
Otherwise $\upsilon_i^{(\mu)}+\upsilon_i^{(n-\mu)} = e_i$. We conclude that:
\[
m_i^{(\mu)}+m_i^{(n-\mu)} = 
\begin{cases}
l_i-1 & \text{for } i\in I^{(\mu)} \\
l_i & \text{otherwise,}
\end{cases}
\]
hence:
\begin{equation}\label{eq:gg}
g_{\mu}(x)\cdot g_{n-\mu}(x) = \frac{\prod_{i=1}^r (x-\rho_i)^{l_i}}{\prod_{i\in I^{(\mu)}}(x-\rho_i)} = \frac{f(x)}{\prod_{i\in I^{(\mu)}}(x-\rho_i)}.
\end{equation}
\par
Going back to the calculation of $d\left(\frac{y^{\mu}}{x^{\nu}g_{\mu}(x)}\right)$  in eq. (\ref{eq:derder}) using eq. (\ref{eq:gg}) we have:
\begin{align*}
d\left(\frac{y^{\mu}}{x^{\nu}g_{\mu}(x)}\right) =& \frac{\prod_{i\in I^{(\mu)}}(x-\rho_i)\left(x\sum_{j\in I^{(\mu)}}\frac{\upsilon_j^{(\mu)}(n,l_j)}{x-\rho_j}-\nu n\right)}{nx^{\nu+1}}\omega_{n-\mu,1}\\
=&\frac{\psi_{\mu,\nu}(x)}{nx^{\nu+1}}\omega_{n-\mu,1}\\
=&\frac{\psi_{\mu,\nu}^{\leq \nu+1}(x)}{nx^{\nu+1}}\omega_{n-\mu,1}-\frac{\psi_{\mu,\nu}^{> \nu+1}(x)}{nx^{\nu+1}}\omega_{n-\mu,1},
\end{align*}
where of course the first summand is in $\Omega_0$, and the second summand is in $\Omega_{\infty}$, as we wanted.
\end{proof}
\section{Artin-Schreier extensions}
\label{sec:AS}
\subsection{Computations}
We consider the curve given as an Artin-Schreier extension of $k(x)$, where $k$ is again an algebraically closed field of characteristic $p\geq 3$, given by an equation of the form 
\[
X:y^p-y=r(x)=\frac{f(x)}{\prod_{i=1}^r (x-\rho_i)^{l_i}}.
\]
Without loss of generality we can assume that $f(x)$ is prime to the denominator of $r(x)$, that $(l_i,p)=1$ for all $i\in \{1,\dots,r\}$, that $x\nmid f(x)$ (i.e. $0$ is not a root of $f(x)$), and that the points above $\infty$ are neither poles nor roots for $y$ (which means that $\deg f(x) = \sum_{i=1}^r l_i =:l$). This means that $\infty$ is not in the branch locus of the cover $X \rightarrow \mathbb{P}^1$ and the points above it are $P_{\infty,1},\ldots,P_{\infty,p}$. For any given Artin-Schreier curve, we can use Möbius transformations 
of the base curve $\mathbb{P}^1$ to make it satisfy the above conditions. The points of $X$ lying above each $x=\rho_i$ are $P_i$ with ramification index $p$.\par
If $0\in \mathbb{P}^1$ is in the branch locus, then $g_0:=1 \text{ and } e_0:=p$. If $0$ is not in the branch locus, then $g_0:=p \text{ and } e_0:=1$.\par
We collect several known facts for the divisors of functions and differentials, see \cite[sec. 3]{Boseck}
\begin{align*}
(x) &= e_0\sum_{j=1}^{g_0} P_{0,j} - \sum_{m=1}^p P_{\infty,m}\\
(y) &= (y)_0 - \sum_{i=1}^r l_i P_i = (f(x))_0 - \sum_{i=1}^r l_i P_i,
\end{align*}
where in the root divisor $(y)_0=(f(x))_0$ of $y$, none of the $P_i$'s appears. Moreover
\begin{align*}
\left(\prod_{i=1}^r (x-\rho_i)^{l_i}\right) &= p\sum_{i=1}^r l_i P_i - l \sum_{m=1}^p P_{\infty,m}\\
(y^{p-1}-1) &= A - (p-1)\sum_{i=1}^r l_i P_i
\end{align*}
and $A$ is the root part of the divisor of $y^{p-1}-1$ where the roots of $f(x)$ appear.\par
We will also need the divisor of the differential $dx$:
\[
(dx) = \sum_{i=1}^{r} (p-1)(l_i+1)P_i - 2\sum_{m=1}^{p} P_{\infty,m},
\]
hence for $\mu$ a positive integer we get
\[
\left(y^{\mu}dx\right) = \mu (y)_0 + \sum_{i=1}^r \big((p-1-\mu)l_i + (p-1)\big)P_i - 2\sum_{m=1}^p P_{\infty,m}.
\]\par
For each $\mu$, $\mu \in \{1,2,\ldots,p-1\}$, we define $m_i^{(\mu)}$ and $\upsilon_i^{(\mu)}$ as the quotient and the residue, respectively, of the following Euclidean division
\[
(p-1-\mu) l_i + (p-1) = m_i^{(\mu)}p + \upsilon_i^{(\mu)}, \text{ with } 0 \leq \upsilon_i^{(\mu)}  \leq p-1.
\]
We set
\[
g_\mu(x) = \prod_{i=1}^r (x-\rho_i)^{m_i^{(\mu)}} \text{ for } \mu \in \{1,2,\ldots,p-1\}
\]
and we have that
$$(g_{\mu}(x)) = p\sum_{i=1}^r m_i^{(\mu)} P_i - t^{(\mu)}\sum_{m=1}^p P_{\infty,m},$$
where 
\[
t^{(\mu)}=\sum_{i=1}^{r}m_i^{(\mu)}.
\]
Now, we compute
\begin{equation*}
\left(\frac{y^{\mu}}{g_{\mu}(x)}dx\right) = \mu (y)_0 + \sum_{i=1}^{r} \upsilon_i^{(\mu)} P_i + (t^{(\mu)} -2)\sum_{m=1}^p P_{\infty,m}.
\end{equation*}
Finally we will need 
\begin{align}\nonumber \label{eq:extrael2}
	\left(g_{p-\mu}(x)y^{\mu-1}\right) =& p\sum_{i=1}^r m_i^{(p-\mu)} P_i - t^{(p-\mu)}\sum_{m=1}^p P_{\infty,m}\\
	& + (\mu-1)(f(x))_0 - \sum_{i=1}^r (\mu-1)l_i P_i\\
\nonumber	=& \sum_{i=1}^r (pm_i^{(p-\mu)}-(\mu-1)l_i)P_i + (\mu-1)(f(x))_0- t^{(p-\mu)}\sum_{m=1}^p P_{\infty,m}
\end{align}
where $pm_i^{(p-\mu)}-(\mu-1)l_i = p-1 - \upsilon_i^{(p-\mu)}\geq 0$.
\subsection{A $k$-vector space basis for $H^1(X, \mathcal{O}_X)$}
As in the Kummer case,
we denote by $\pi:X \rightarrow \mathbb{P}^1$ a covering map from $X$ to $\mathbb{P}^1$. For every point $a \in \mathbb{P}^1$ we define $U_a := X \backslash \pi^{-1}(a)$. Observe that by the calculation (\ref{eq:extrael2}) the elements
\[
\frac{g_{p-\mu}(x)y^{\mu-1}}{x^{\nu}} \text{ for } 1 \leq \mu \leq p-1,\text{ s.t. } t^{(p-\mu)}\geq 2, 1\leq \nu \leq t^{(p-\mu)} - 1
\]
have poles only above zero and infinity, i.e. they are in $\mathcal{O}_X(U_0\cap U_{\infty})$.
By  identifying  the first sheaf cohomology group of $\mathcal{O}_X$ with the corresponding first Čech cohomology group, then we get that 
\[
H^{1}(X, \mathcal{O}_X) = \frac{  \mathcal{O}_X(U_0 \cap U_\infty) }{
	\{ f_0- f_\infty: 
	f_0 \in \mathcal{O}_X(U_0), f_\infty \in \mathcal{O}_X(U_\infty) \}
}.
\]
From Serre duality, this $H^1(X, \mathcal{O}_X)$ seen as a $k$-vector space, is dual to the $k$-vector space $H^0(X,\Omega_X)$. From the calculations of the previous paragraph and from \cite{Boseck} we can see that a basis of $H^0(X,\Omega_X)$ is given by the differentials 
\begin{align*}
\label{differential_basis}
\omega_{\mu,\nu}:=x^{\nu-1}\frac{{y}^{\mu}}{{g_{\mu}(x)}}dx, 1\leq \mu \leq p-1, \text{ s.t. } t^{(\mu)}\geq 2, 1 \leq \nu \leq t^{(\mu)}-1.
\end{align*}
Also, the following linear equivalence classes
\begin{align*}
h_{\mu,\nu} := \left[\frac{g_{p-\mu}(x)y^{\mu-1}}{x^{\nu}}\right] \text{ with }1\leq \mu \leq p-1, \text{ s.t. } t^{(p-\mu)}\geq 2, 1 \leq \nu \leq t^{(p-\mu)}-1.
\end{align*}
are well defined, linearly independent elements of $H^1(X, \mathcal{O}_X)$. Since they are in bijective correspondence with the elements $\omega_{\mu,\nu}$, we know that they form a basis (the $k$-vector space dimension of $H^0(X,\Omega_X)$ is equal to the $k$-vector space dimension of $H^1(X,\mathcal{O}_X)$).\par
\begin{remark}
	Note that the condition $t^{(\mu)}\geq 2$ means that $\mu < p-1$, because for $\mu=p-1$ we have $m_i^{(\mu)} = 0$ for all $i$. For the same reason the condition $t^{(p-\mu)}\geq 2$ means that $\mu > 1$.
\end{remark}
As in the section of the Kummer case, using the map
\begin{align*}
\langle\cdot,\cdot\rangle:H^0(X,\Omega_X) \times H^1(X, \mathcal{O}_X) & \longrightarrow  k \\
\langle\omega, [f]\rangle & \longmapsto  \sum_{P\in \pi^{-1}(\infty)}\mathrm{Res}_P(f \omega)
\end{align*}
we will show that:
\begin{proposition}
	The bases of $H^0(X,\Omega_X)$ and $H^1(X,\mathcal{O}_X)$ that we have written above are dual to each other.
\end{proposition}
\begin{proof}
	We calculate the pairing between elements of the above given bases:
	\begin{align*}
	 \langle\omega_{\mu_1,\nu_1},h_{\mu_2,\nu_2}\rangle=&\sum_{P\in \pi^{-1}(\infty)}\mathrm{Res}_P\left(\frac{g_{p-\mu_2}(x)y^{\mu_2-1}}{x^{\nu_2}}\frac{x^{\nu_1-1}y^{\mu_1}}{g_{\mu_1}(x)}dx\right)\\
	 =&\sum_{P\in\pi^{-1}(\infty)}\mathrm{Res}_P\left(\frac{y^{\mu_1+\mu_2-1}g_{p-\mu_2}(x)x^{\nu_1-\nu_2}}{g_{\mu_1}(x)}\frac{dx}{x}\right)\\
	 =&\mathrm{Res}_{\infty}\left(\text{Tr}_{F/E}\left(y^{\mu_1+\mu_2-1}\right)\frac{g_{p-\mu_2}(x)x^{\nu_1-\nu_2}}{g_{\mu_1}(x)}\frac{dx}{x}\right)
	\end{align*}\par
	We have to calculate the trace $y^{\mu_1+\mu_2-1}$ for the various values of $\mu_1$ and $\mu_2$. Notice that $1\leq \mu_1+\mu_2-1\leq 2(p-1)-1$, so we will write the trace of $y^k$ for $k\in \{1,\dots,2(p-1)-1\}$:
	\begin{align*}
	\text{Tr}_{F/E}(y^k) =& \sum_{i=0}^{p-1} (y+i)^k = y^k + \sum_{i=1}^{p-1} (y+i)^k\\
	=& y^k + \sum_{i=1}^{p-1} \sum_{j=0}^{k} \binom{k}{j} y^{k-j}i^j\\
	=& y^k + \sum_{j=0}^{k} \binom{k}{j} y^{k-j}\sum_{i=1}^{p-1}i^j\\
	=& y^k + y^k\left(\sum_{i=1}^{p-1}i^0\right) + \sum_{j=1}^k \binom{k}{j} y^{k-j}\sum_{i=1}^{p-1}i^j\\
	=& \sum_{j=1}^k \binom{k}{j} y^{k-j}\sum_{i=1}^{p-1}i^j
	\end{align*}
	since $\sum_{i=1}^{p-1}i^0 = p-1 = -1$ in characteristic $p$, and this cancels out with the first $y^k$.\par
	Let's focus on the inner sum of the above expression: For $j\in\{1,\dots,p-2,p,\dots,2(p-1)-1\}$, we have: 
	$$\sum_{i=1}^{p-1} i^j = \sum_{t\in \mathbb{F}_p^{\times}} t^j = \sum_{l=0}^{p-2} (g^l)^j$$
	where $g$ is a generator of $\mathbb{F}_p^{\times}$. Since $j\in\{1,\dots,p-2,p,\dots,2(p-1)-1\}$ we have $g^j\neq 1$, and thus
	$$\sum_{i=1}^{p-1} i^j = \sum_{l=0}^{p-2} (g^j)^l = \frac{1-(g^j)^{p-1}}{1-g^j} = \frac{1-(g^{p-1})^j}{1-g^j} = 0$$
	where the last sum is the sum of terms of a geometric progression, and of course $g^{p-1} = 1$ for $g\in \mathbb{F}_p^{\times}$.\par
	On the other hand for $j=p-1$:
	$$\sum_{i=1}^{p-1}i^j = \sum_{t\in \mathbb{F}_p^{\times}} t^{p-1} = p-1 = -1 \text{ in characteristic } p.$$
	However for $k\geq p$, the binomial coefficient $\binom{k}{p-1}$ is divisible by $p$, hence it is zero in characteristic $p$.\par
	From all the above we conclude that:
	\[
	\text{Tr}_{F/E}(y^k) = 
	\begin{cases}
	-1 & \text{for } k=p-1\\
	0 & \text{for } k\in \{1,\dots,p-2,p,\dots,2(p-1)-1\}.
	\end{cases}
	\]\par
	Returning to the calculation of the pairing, we see that the result is non-zero only for $\mu_1+\mu_2-1 = p-1$, i.e. for $\mu_1=p-\mu_2$. In that case:
	\begin{align*}
	\langle\omega_{p-\mu_2,\nu_1},h_{\mu_2,\nu_2}\rangle =& \mathrm{Res}_{\infty}\left((-1)\frac{g_{p-\mu_2}(x)x^{\nu_1-\nu_2}}{g_{p-\mu_2}(x)}\frac{dx}{x}\right)\\
	=& \mathrm{Res}_{\infty}\left(\left(\frac{1}{x}\right)^{\nu_2-\nu_1}\frac{d\left(\frac{1}{x}\right)}{\left(\frac{1}{x}\right)}\right).
	\end{align*}
	From the properties of the residues, the last residue from above is non-zero if and only if $\nu_1=\nu_2$, in which case it is equal to $1$, and this completes our proof.
\end{proof}
\subsection{A $k$-vector space basis for $H^1_{dR}(X/k)$}
Similarly to the case of a Kummer extension we can compute $H^1_{dR}(X/k)$ as follows:
\begin{center}
	\resizebox{.99\hsize}{!}{$\frac{\{(\omega_0,\omega_{\infty},f_{0\infty}) \in \Omega_X(U_0) \times \Omega_X(U_{\infty}) \times \mathcal{O}_X(U_0\cap U_{\infty}) \mid df_{0\infty}=\omega_0|_{U_0\cap U_{\infty}} - \omega_{\infty}|_{U_0\cap U_{\infty}} \} }{
			\{ (df_0, df_{\infty}, f_0|_{U_0\cap U_{\infty}} - f_{\infty}|_{U_0\cap U_{\infty}}) \mid f_0\in \mathcal{O}_X(U_0), f_{\infty}\in \mathcal{O}_X(U_{\infty})\}}.$}	
\end{center}
Again, we have canonical maps
\begin{align*}
i: H^0(X,\Omega_X) & \rightarrow H^1_{dR}(X/k) \\
\omega & \mapsto [(\omega|_{U_0},\omega|_{U_{\infty}},0)]
\end{align*}
and
\begin{align*}
p: H^1_{dR}(X/k) & \rightarrow H^1(X,\mathcal{O}_X) \\
[(\omega_0, \omega_{\infty}, f_{0\infty})] & \mapsto [f_{0\infty}].
\end{align*}
Same as in the previous section, from Proposition 4.1.2 of Tait's dissertation (\cite{TaitPhd}) we get that the following short sequence is exact:
\[
0\rightarrow H^0(X,\Omega_X) \xrightarrow{i} H^1_{dR}(X/k) \xrightarrow{p} H^1(X,\mathcal{O}_X) \rightarrow 0
\]
hence the basis of the $k$-vector space $H^1_{dR}(X/k)$ can be calculated from the bases of $H^0(X,\Omega_X)$ and $H^1(X,\mathcal{O}_X)$.\par
\begin{theorem}\label{Artin-Schreier_basis}
	Let 
	$$\varphi_{\mu,\nu}(x):= xg_{p-\mu}'(x)g_{\mu-1}(x) - \nu g_{p-\mu}(x)g_{\mu-1}(x) \in k[x]$$
	for $1\leq \mu\leq p-1$ s.t. $t^{(p-\mu)}\geq 2$ and $1\leq \nu\leq t^{(p-\mu)}-1$. Also let
	$$\psi(x):=f(x)\sum_{i=1}^r l_i\left(\prod_{j=1,j\neq i}^r (x-\rho_j)\right) - f'(x)\prod_{i=1}^r(x-\rho_i)\in k[x].$$
	Let
	$$\omega_{\mu}:=\frac{(\mu-1)g_{p-\mu}(x)y^{\mu-2}}{\prod_{i=1}^r (x-\rho_i)^{l_i+1}}dx$$
	for $1\leq\mu\leq p-1 \text{ s.t. } t^{(p-\mu)}\geq 2$. Let us denote
	$$\tilde{h}_{\mu,\nu}:= \frac{g_{p-\mu}(x)y^{\mu-1}}{x^{\nu}}.$$
	The elements
	\begin{gather*}
	a_{\mu,\nu} := \left[\left(\frac{\varphi_{\mu,\nu}^{<\nu+1}(x)}{x^{\nu+1}}\omega_{\mu-1,1}+\frac{\psi^{<\nu}(x)}{x^{\nu}}\omega_{\mu},\frac{\varphi_{\mu,\nu}^{\geq
	\nu+1}(x)}{x^{\nu+1}}\omega_{\mu-1,1}+\frac{\psi^{\geq\nu}(x)}{x^{\nu}}\omega_{\mu},\tilde{h}_{\mu,\nu}\right)\right]\\
	\text{ for } 1\leq \mu \leq p-1 \text{ s.t. } t^{(p-\mu)}\geq 2 \text{ and } 1\leq \nu \leq t^{(p-\mu)}-1
	\end{gather*}
	together with the elements
	\begin{gather*}
	\delta_{\mu,\nu} := \left[\left(x^{\nu-1}\frac{y^{\mu}}{g_{\mu}(x)}dx,x^{\nu-1}\frac{y^{\mu}}{g_{\mu}(x)}dx,0\right)\right] \\
	\text{ for } 1\leq \mu \leq p-1 \text{ s.t. } t^{(\mu)}\geq 2 \text{ and } 1\leq \nu \leq t^{(\mu)}-1
	\end{gather*}
	form a basis for $H^1_{dR}(X/k)$ as a $k$-vector space.
\end{theorem}
\begin{proof}
	Similarly to the corresponding theorem about Kummer extensions, we just need to show that the $a_{\mu,\nu}$'s are well-defined elements of $H^1_{dR}(X/k)$.\par
	First we write down the differential of $\tilde{h}_{\mu,\nu}$:
	\begin{align*}
	d\left(\frac{g_{p-\mu}(x)y^{\mu-1}}{x^{\nu}}\right) 
	=& 
	\frac{(g_{p-\mu}'(x)y^{\mu-1} \!-\nu g_{p-\mu}(x)y^{\mu-1}x^{-1})dx\!+(\mu-1)g_{p-\mu}(x)y^{\mu-2}dy}{x^{\nu}}\\
	=& \frac{xg_{p-\mu}'(x)g_{\mu-1}(x) - \nu g_{p-\mu}(x)g_{\mu-1}(x)}{x^{\nu+1}}\left(\frac{y^{\mu-1}}{g_{\mu-1}(x)}dx\right)\\ 
	&+ \frac{(\mu-1)g_{p-\mu}(x)y^{\mu-2}dy}{x^{\nu}}\\
	=& \frac{\varphi_{\mu,\nu}(x)}{x^{\nu+1}}\omega_{\mu-1,1} + \frac{(\mu-1)g_{p-\mu}(x)y^{\mu-2}}{x^{\nu}}dy.
	\end{align*}\par
	For the differential of $y$ we can use the defining equation of our curve, since we are in characteristic $p$:
	\begin{align*}
	d(y^p-y) =& d\left(\frac{f(x)}{\prod_{i=1}^r(x-\rho_i)^{l_i}}\right)\\
	\Rightarrow -dy =& \frac{f'(x) - f(x)\sum_{i=1}^r\frac{l_i}{x-\rho_i}}{\prod_{i=1}^r(x-\rho_i)^{l_i}}\\
	\Rightarrow dy = & \frac{f(x)\sum_{i=1}^r l_i\left(\prod_{j=1,j\neq i}^r(x-\rho_j)\right) - f'(x)\prod_{i=1}^r(x-\rho_i)}{\prod_{i=1}^r (x-\rho_i)^{l_i+1}}\\
	\Rightarrow dy = & \frac{\psi(x)}{\prod_{i=1}^r (x-\rho_i)^{l_i+1}}.
	\end{align*}\par
	Thus, we end up with:
	$$d\left(\frac{g_{p-\mu}(x)y^{\mu-1}}{x^{\nu}}\right) = \frac{\varphi_{\mu,\nu}(x)}{x^{\nu+1}}\omega_{\mu-1,1} + \frac{\psi(x)}{x^{\nu}}\omega_{\mu}.$$
	We will show that for all $\mu\in \{1,\dots,p-1\}$ s.t. $t^{(p-\mu)}\geq 2$, $\omega_{\mu}$ is either holomorphic or has pole of order one at each point above $\infty$. Then it's easily seen that the $a_{\mu,\nu}$'s are well-defined, so the proof will be over.\par
	Notice that $\omega_{\mu}$ can have poles only at the points above $\infty$ (we'll show that when that happens the order of the pole is one) or at the $P_i$'s (we'll show that this never happens), since $f(x)$ never appears in the denominator - remember that $\mu\geq 2$ from the remark we made after writing the basis of $H^1(X,\mathcal{O}_X)$ above.\par
	Firstly, let $P_{\infty,m}$ be any point above $\infty$. Then:
	\begin{align*}
	v_{P_{\infty,m}}(\omega_{\mu}) =& v_{P_{\infty,m}}(g_{p-\mu}(x)) + v_{P_{\infty,m}}(dx) + v_{P_{\infty,m}}\left(\prod_{i=1}^r(x-\rho_i)^{l_i+1}\right)\\
	=& -t^{(p-\mu)}-2+\sum_{i=1}^r(l_i+1)\\
	=& -2 +\sum_{i=1}^r (l_i - m_i^{(p-\mu)} + 1),
	\end{align*}
	where for all $i\in \{1,\dots,r\}$
	$$(p-1)(l_i+1) - (p-\mu) l_i = pm_i^{(p-\mu)}+\upsilon_i^{(p-\mu)},$$
	which leads to
	$$(p-1)(l_i-m_i^{(p-\mu)}+1) = (p-\mu)l_i + m_i^{(p-\mu)} + \upsilon_i^{(p-\mu)}.$$
	The right hand side of the last equation is strictly positive, so the same is true for the left hand side, and more specifically
	$$l_i-m_i^{(p-\mu)}+1 > 0, \text{ for all } i\in \{1,\dots,r\}.$$
	Therefore if $r>1$, then $\sum_{i=1}^r(l_i-m_i^{(p-\mu)} +1)\geq 2$ and $\omega_{\mu}$ has no pole at any point above $\infty$. If $r=1$, then $l_1-m_1^{(p-\mu)}+1\geq 1$ and $\omega_{\mu}$ is either holomorphic (if $l_1>m_1^{(p-\mu)}$) or has a pole at each $P_{\infty,m}$ of order one (if $l_1=m_1^{(p-\mu)}$; for example $p=7,\mu=6,l_1=2$ give $m_1^{(1)}=2=l_1$).\par
	Now let's check whether $\omega_{\mu}$ has pole at any $P_i$:
	\begin{align*}
	v_{P_i}(\omega_{\mu})=& v_{P_i}(g_{p-\mu}(x)) + v_{P_i}(y^{\mu-2}) + v_{P_i}(dx) - v_{P_i}\left(\prod_{i=1}^r(x-\rho_i)^{l_i+1}\right)\\
	=& pm_i^{(p-\mu)} - (\mu-2)l_i + (p-1)(l_i+1) - p(l_i+1)\\
	=& (\mu-1)l_i + p-1 - \upsilon_i^{(p-\mu)} - (\mu-2)l_i - l_i - 1\\
	=& p-2-\upsilon_i^{(p-\mu)}.
	\end{align*}
	If $\upsilon_i^{(p-\mu)} = p-1$, then $(\mu-1)l_i\equiv 0\mod p$ with $(p,l_i)=1$ by assumption, so $\mu=1$, but that's not allowed from the remark we made after writing the basis for $H^1(X,\mathcal{O}_X)$. Hence $\upsilon_i^{(p-\mu)}\leq p-2$, which implies that $v_{P_i}(\omega_{\mu})\geq 0$, as we wanted.
\end{proof}

 \def\cprime{$'$}

\end{document}